\documentclass[12pt,reqno]{amsart}
\usepackage{amsmath, amsthm, amssymb, stmaryrd}
\usepackage{hyperref}
\usepackage{enumerate}
\usepackage{url}

\let\OLDthebibliography\thebibliography
\renewcommand\thebibliography[1]{
  \OLDthebibliography{#1}
  \setlength{\parskip}{0pt}
  \setlength{\itemsep}{0pt plus 0.3ex}
}

\usepackage{mathtools}

\usepackage{multirow, array}
\usepackage{placeins}

\usepackage{caption} 
\advance \topmargin by -\headheight
\advance \topmargin by -\headsep
     
\evensidemargin \oddsidemargin
\newtheorem{thm}{Theorem}[section]
\newtheorem{lemma}[thm]{Lemma}

\newtheorem{cor}[thm]{Corollary}
\newtheorem{conj}[thm]{Conjecture}

\theoremstyle{definition}

\theoremstyle{remark}

\numberwithin{equation}{section}

\newcommand{\mmod}[1]{{\,\,\mathrm{mod}\,\,#1}}

\newcommand*\wrapletters[1]{\wr@pletters#1\@nil}
\def\wr@pletters#1#2\@nil{#1\allowbreak\if&#2&\else\wr@pletters#2\@nil\fi}

\def\alp{{\alpha}} 
\def\bet{{\beta}}  
\def\gam{{\gamma}} 
\def\del{{\delta}}

\def\lam{{\lambda}} \def\Lam{{\Lambda}}

\def\eps{\varepsilon}

\def\le{\leqslant} \def\ge{\geqslant}

\def \bN {\mathbb N}

\def \bQ {\mathbb Q}
\def \bR {\mathbb R}
\def \bZ {\mathbb Z}

\def \cB {\mathcal B}

\def \cS {\mathcal S}

\def \det {\mathrm{det}}

\begin{document}
\title[Bohr sets and diophantine approximation]{Bohr sets and multiplicative diophantine approximation}
\author[Sam Chow]{Sam Chow}
\address{The Mathematical Sciences Research Institute, 17 Gauss Way, Berkeley, CA 94720-5070, United States}
\address{Department of Mathematics, University of York,
Heslington, York, YO10 5DD, United Kingdom}
\email{sam.chow@york.ac.uk}
\thanks{}
\date{}
\begin{abstract} In two dimensions, Gallagher's theorem is a strengthening of the Littlewood conjecture that holds for almost all pairs of real numbers. We prove an inhomogeneous fibre version of Gallagher's theorem, sharpening and making unconditional a result recently obtained conditionally by Beresnevich, Haynes and Velani. The idea is to find large generalised arithmetic progressions within inhomogeneous Bohr sets, extending a construction given by Tao. This precise structure enables us to verify the hypotheses of the Duffin--Schaeffer theorem for the problem at hand, via the geometry of numbers.
\end{abstract}
\maketitle

\section{Introduction}
\label{intro}

\subsection{Results}

A famous conjecture by Littlewood \cite[\S 4.4]{BRV2016} asserts that if $\alp, \bet \in \bR$ then 
\[
\liminf_{n \to \infty} n \| n \alp \| \cdot \| n \bet \| = 0.
\]
However, Gallagher's theorem \cite{Gal1962} implies that for almost all pairs $(\alp, \bet) \in \bR^2$ the stronger statement
\begin{equation} \label{Gallagher}
\liminf_{n \to \infty} n (\log n)^2 \| n \alp \| \cdot \| n \bet \| = 0
\end{equation}
is valid. Beresnevich, Haynes and Velani \cite[Theorem 2.1 and Remark 2.4]{BHV2016} recently showed that for any $\alp \in \bR$ the statement \eqref{Gallagher} holds for almost all $\bet \in \bR$. On inhomogeneous fibres, they were able to establish the following conditional outcome \cite[Theorem 2.4]{BHV2016}.

\begin{thm}[Beresnevich--Haynes--Velani] \label{conditional} 
Let $\alp, \gam \in \bR$, and assume that $\alp$ is irrational and not Liouville. Further, assume the Duffin--Schaeffer conjecture (Conjecture \ref{DSconj}). Then for almost all $\bet \in \bR$ we have
\begin{equation} \label{Gallagher2}
\liminf_{n \to \infty} n (\log n)^2 \| n \alp - \gam\| \cdot \| n \bet \| = 0.
\end{equation}
\end{thm}

Here we recall that an irrational real number $\alpha$ is \emph{Liouville} if
\[
\liminf_{n \to \infty} n^w \| n \alpha \| = 0
\]
for all $w > 0$. We prove the result unconditionally.

\begin{thm} \label{MainThm} Let $\alp, \gam \in \bR$, and assume that $\alp$ is irrational and not Liouville. Then for almost all $\bet \in \bR$ we have \eqref{Gallagher2}.
\end{thm}

The assumption that $\alp$ is not Liouville is mild. Indeed, it follows from the Jarn\'ik--Besicovitch theorem \cite[Theorem 3.2]{BRV2016} that the set of Liouville numbers has Hausdorff dimension zero. That said, we believe the Liouville case to be more demanding, owing to the erratic behaviour of certain sums --- see \cite[Theorem 1.8]{BHV2016} and its proof. We expect the Liouville case to require a somewhat different method, and hope to address it in future work.

Note that the case $\gam = 0$ of Theorem 1.2 follows without any assumption on $\alp \in \bR$, for if $\alp$ is rational or Liouville then we plainly have the stronger result
\[
\liminf_{n \to \infty} n^2 \|n \alp \| = 0.
\]

We shall in fact prove the following strengthening of Theorem \ref{MainThm}.

\begin{thm} \label{main} 
Let $\alp, \gam \in \bR$, and assume that $\alp$ is irrational and not Liouville. Let $\psi: \bN \to \bR_{\ge 0}$ be a decreasing function such that 
\begin{equation} \label{divergence}
\sum_{n=1}^\infty \psi(n) \log n = \infty.
\end{equation}
Then for almost all $\bet \in \bR$ there exist infinitely many $n \in \bN$ such that
\[
\| n \alp - \gam \| \cdot \| n \bet \| < \psi(n).
\]
\end{thm}

We know at least in some cases that the assumption \eqref{divergence} is necessary, as demonstrated in \cite[Corollary 2.1]{BHV2016} given below. Here we recall that $\alp \in \bR$ is \emph{badly approximable} if 
\[
\liminf_{n \to \infty} n \| n \alp \| > 0.
\]

\begin{thm}[Beresnevich--Haynes--Velani] \label{convergence} Let $\alp$ be a badly approximable number, let $\del \in \bR$, and let $\psi: \bN \to \bR_{\ge 0}$ be a decreasing function such that
\[
\sum_{n=1}^\infty \psi(n) \log n < \infty.
\]
Then for almost all $\bet \in \bR$ the inequality
\[
\| n \alp \| \cdot \| n \bet - \del \| < \psi(n)
\]
holds for only finitely many $n \in \bN$.
\end{thm}

In particular, if we set $\gam = 0$ in Theorem \ref{divergence} and $\del = 0$ in Theorem \ref{convergence}, we obtain the following dichotomy.

\begin{cor}
Let $\alp$ be a badly approximable number, and let $\psi: \bN \to \bR_{\ge 0}$ be a decreasing function. Then the measure of the set
\[
\{ \bet \in [0,1]: \| n \alp \| \cdot \| n \bet \| < \psi(n) \text{ for infinitely many } n \in \bN \}
\]
is
\[
\begin{cases} 
0, &\text{if } \sum_{n=1}^\infty \psi(n) \log n < \infty \\
1, &\text{if } \sum_{n=1}^\infty \psi(n) \log n = \infty.
\end{cases}
\]
\end{cor}

Some readers may wonder about the corresponding Hausdorff theory. This appears to be less interesting: as discussed in \cite[\S1]{BV2015}, genuine `fractal' Hausdorff measures are insensitive to the multiplicative nature of such problems.

If we knew an inhomogeneous version of the Duffin--Schaeffer theorem, then the following assertion would follow from our method.

\begin{conj} \label{cond} Let $\alp, \gam, \del \in \bR$, and assume that $\alp$ is irrational and not Liouville.  Let $\psi: \bN \to \bR_{\ge 0}$ be a decreasing function satisfying \eqref{divergence}. Then for almost all $\bet \in \bR$ there exist infinitely many $n \in \bN$ such that
\[
\| n \alp - \gam \| \cdot \| n \bet - \del \| < \psi(n).
\]
\end{conj}

It would follow, for instance, if we knew the following.

\begin{conj} [Inhomogeneous Duffin--Schaeffer theorem]
Let $\del \in \bR$ and $\Phi: \bN \to \bR_{\ge 0}$. Assume
\begin{equation} \label{DShyp}
\sum_{n=1}^\infty \frac{\varphi(n)}n \Phi(n) = \infty
\end{equation}
and 
\begin{equation} \label{additional}
\limsup_{N\to \infty} \Bigl(\sum_{n \le N} \frac{\varphi(n)}n \Phi(n) \Bigr)
\Bigl(\sum_{n\le N} \Phi(n)\Bigr)^{-1} > 0.
\end{equation}
Then for almost all $\bet \in \bR$ there exist infinitely many $n \in \bN$ such that
\[
\| n \bet - \del \| < \Phi(n).
\]
\end{conj}

Here $\varphi$ is the Euler totient function. The statement above should be compared to Theorem \ref{DSthm}. There is little consensus over what the `right' statement of the inhomogeneous Duffin--Schaeffer theorem should be, but there is some relevant discussion in \cite{Ram2016}. In the inhomogeneous setting, we do not at present even have an analogue of Gallagher's zero-one law \cite{Gal1961}. Our statement is partly motivated by a random model being developed by Ram\'{i}rez \cite{Ram2017}; such a framework could potentially transfer to the inhomogeneous setting.

\subsection{The method}

Consider the auxiliary approximating function $\Phi = \Phi_\alp^\gam$ given by
\begin{equation} \label{PhiDef}
\Phi(n) = \frac{\psi(n)}{\|n \alp - \gam \|}.
\end{equation}
The conclusion of Theorem \ref{main} is equivalent to the assertion that for almost all $\bet \in \bR$ there exist infinitely many $n \in \bN$ such that
\[
\| n \bet \| < \Phi(n).
\]
If $\Phi$ were monotonic, then Khintchine's theorem \cite[Theorem 2.3]{BRV2016} would be a natural approach.

\begin{thm}[Khintchine's theorem]
Let $\Phi: \bN \to \bR_{\ge 0}$ be a monotonic function. Then the measure of the set
\[
\{ \bet \in [0,1]: \| n \bet \| < \Phi(n) \text{ for infinitely many } n \in \bN \}
\]
is
\[
\begin{cases} 
0, &\text{if } \sum_{n=1}^\infty \Phi(n) < \infty \\
1, &\text{if } \sum_{n=1}^\infty \Phi(n) = \infty.
\end{cases}
\]
\end{thm}

For any $n \in \bN$ the function $\bet \mapsto \| n \bet \|$ is periodic modulo 1, so Khintchine's theorem implies that if $\Phi$ is monotonic and $\sum_{n=1}^\infty \Phi(n) = \infty$ then for almost all $\bet \in \bR$ the inequality $\| n \bet \| < \Phi(n)$ holds for infinitely many $n \in \bN$. The specific function $\Phi$ defined in \eqref{PhiDef} is very much \textbf{not} monotonic, so our task is much more demanding. Beresnevich--Haynes--Velani took a Duffin--Schaeffer conjecture \cite[p. 255]{DS1941} approach to establish Theorem \ref{conditional}.

\begin{conj}[Duffin--Schaeffer conjecture]\label{DSconj}
Let $\Phi: \bN \to \bR_{\ge 0}$ satisfy \eqref{DShyp}. Then for almost all $\bet \in \bR$ the inequality
\[
| n \bet - r | < \Phi(n)
\]
holds for infinitely many coprime pairs $(n,r) \in \bN \times \bZ$.
\end{conj}

We note for comparison to Khintchine's theorem that if $\Phi$ is monotonic then the condition \eqref{DShyp} is equivalent to $\sum_{n=1}^\infty \Phi(n) = \infty$. Beresnevich--Haynes--Velani obtained their conditional result by confirming the hypothesis \eqref{DShyp} for the specific function \eqref{PhiDef}. 

The Duffin--Schaeffer conjecture has stimulated research in diophantine approximation for decades, and is still wide open. There has been some progress, including the Erd\H{o}s--Vaaler theorem \cite[Theorem 2.6]{Har1998}, as well as \cite{HPV2012, BHHV2013} and \cite{Aist2014}. For us, the most relevant partial result is  the Duffin--Schaeffer theorem \cite[Theorem 2.5]{Har1998}.

\begin{thm}[Duffin--Schaeffer theorem] \label{DSthm}
Conjecture \ref{DSconj} holds under the additional hypothesis \eqref{additional}.
\end{thm}

If $\Phi$ were monotonic or supported on primes then the hypothesis \eqref{additional} would pose no difficulties \cite[p. 27]{Har1998}, but in general this hypothesis is quite unwieldy. There have been very few genuinely different examples in which the Duffin--Schaeffer theorem has been applied so, \emph{a priori}, our strategy may come across as being particularly ambitious. One other example of an application of the Duffin--Schaeffer theorem is \cite[Theorem 2.3]{BHV2016}, which we shall mention again in due course.

We tame our auxiliary approximating function $\Phi$ by restricting its support to the complement of a `poorly-behaved' set $B$, giving rise to a modified auxiliary approximating function $\Psi = \Psi_\alp^\gam$ --- see \S\S \ref{fractional} and \ref{DS}. By partial summation, we are led to study the sums
\[
\sum_{\substack{n \le N \\ n \notin B}} \frac1{\| n \alp - \gam\|}, \qquad
\sum_{\substack{n \le N \\ n \notin B}} \frac{\varphi(n)}{n \| n \alp - \gam\|}.
\]
Specifically, we require sharp upper bounds for the first sum and sharp lower bounds for the second. The former result follows rather quickly from the work of Beresnevich, Haynes and Velani. For the latter, we ultimately require sharp lower bounds for the average of the Euler totient function $\varphi$ over inhomogeneous Bohr sets
\[
N_\gam(\alp, \rho) = \{ n \in \bN: n \le N, \| n \alp - \gam \| \le \rho \}.
\]

The point is that one needs to understand the structure of these Bohr sets. Beresnevich, Haynes and Velani investigated this structure using the Ostrowski expansion \cite[\S 3]{BHV2016}, resulting in a `gaps lemma' \cite[Lemma 5.1]{BHV2016}. Thus, admittedly, those authors had already provided fairly precise information concerning the structure of $N_\gam(\alp, \rho)$. Nonetheless, it is not at all clear whether this knowledge suffices for a rigorous averaging of the totient function over a Bohr set.

As it were, Bohr sets have been studied in other areas of mathematics, and have been particularly useful in additive combinatorics \cite[\S 4.4]{TV2006}. In a blog post \cite{TaoPost}, Tao explains a correspondence between Bohr sets and generalised arithmetic progressions, and even uses it to provide a possible strategy for proving the Littlewood conjecture. The idea is that there should be generalised arithmetic progressions $P$ and $P'$, of comparable size, for which
\[
P \subseteq N_\gam(\alp, \rho) \subseteq P'.
\]

For us, it is the first inclusion that is important, since it is lower bounds for averages of $\varphi$ that we seek. For some $b, A_1,A_2, N_1, N_2 \in \bN$, we will have a two-dimensional arithmetic progression
\[
P = \{ b+ A_1 n_1 + A_2 n_2 : 1 \le n_1 \le N_1, 1\le n_2 \le N_2 \}.
\]
Tao's construction is for the homogeneous case $\gam = 0$, and we succeed in extending it to the inhomogeneous case in certain ranges, using the assumption that $\alp$ is not Liouville. The idea is to separate the task into homogeneous and inhomogeneous parts, handling the inhomogeneous part using the three distance theorem \cite{MK1998}.

At this stage it should be fairly intuitive that we have enough structure to prove that the totient function averages well: we need to show that $\varphi(n)/n \gg 1$ on average over a generalised arithmetic progression. We use the AM--GM inequality \cite[Ch. 2]{Ste2004} to go from considering an arithmetic mean to considering a geometric mean. This enables us to exploit the fact that $\varphi(n)/n$ is a multiplicative arithmetic function, and to thus separate the contribution from each prime. We are thereby able to reduce the problem to counting lattice points, and to finish the proof.

We briefly address an important technical finesse. The sum
\[
R_N(\alp, \gam) := \sum_{n \le N} \frac1{\| n \alp - \gam\|}
\]
was deeply investigated in \cite{BHV2016}. A recurring theme in that analysis was that a small number of terms, perhaps a single term, could greatly affect the sum. By taming our auxiliary approximating function, as described above, we are able to obtain much better control over the relevant sums. Indeed, in \cite{BHV2016} we see that $R_N(\alp, \gam)$ can behave erratically. Consequently, those authors are only able to solve the problem unconditionally for almost all $\alp$ --- see \cite[Theorem 2.3]{BHV2016}.

We end this discussion with a challenge related to the previous paragraph. In the course of our proof, we establish the Duffin--Schaeffer conjecture for a class of functions, namely those modified auxiliary approximating functions of the shape $\Psi = \Psi_\alp^\gam$, where $\alp$ is irrational and not Liouville. The task of proving the Duffin--Schaeffer conjecture for $\Phi = \Phi_\alp^\gam$, however, remains largely open.

\subsection{Organisation}

In \S \ref{Bohr}, we provide a large two-dimensional arithmetic progression within $N_\gam(\alp, \rho)$. For homogeneous Bohr sets, the construction was given on Tao's blog \cite{TaoPost}. We provide a much simplified account in a significant range in the non-Liouville case, and also extend to the inhomogeneous setting. Then, in \S \ref{average}, we use the generalised arithmetic progression structure together with the geometry of numbers, to show that $\varphi(n)/n \gg 1$ on average over the Bohr set. Our estimate enables us to study the pertinent sums involving $\| n \alp - \gam\|^{-1}$ in \S \ref{fractional}, and to then confirm the hypotheses of the Duffin--Schaeffer theorem in \S \ref{DS}, thereby establishing Theorem \ref{main}. 

\subsection{Notation}

We use Landau and Vinogradov notation: for functions $f$ and positive-valued functions $g$, we write $f \ll g$ or $f = O(g)$ if there exists a constant $C$ such that $|f(x)| \le C g(x)$ for all $x$. Further, we write $f \asymp g$ if $f \ll g \ll f$. If $S$ is a set, we denote the cardinality of $S$ by $|S|$ or $\# S$. The symbol $p$ is reserved for primes. The pronumeral $N$ denotes a positive integer, sufficiently large in terms of constants. When $x \in \bR$, we write $\| x \|$ for the distance from $x$ to the nearest integer.

\subsection{Acknowledgments}

The author was supported by EPSRC Programme Grant EP/J018260/1. The author was supported by the National Science Foundation under Grant No. DMS-1440140 while in residence at the Mathematical Sciences Research Institute in Berkeley, California, during the Spring 2017 semester. Thanks to Victor Beresnevich, Andrew Pollington, Carl Pomerance, Terence Tao and Sanju Velani for valuable suggestions.

\section[Bohr sets]{Bohr sets and generalised arithmetic progressions}
\label{Bohr}

In this section, we study the inhomogeneous one-dimensional Bohr sets $N_\gam(\alp,\rho)$, showing that they contain large two-dimensional arithmetic progressions. The reader may consult \cite{TaoPost} for a more general discussion of the homogeneous case, or \cite[\S 4.4]{TV2006} for a similar theory that is more classical. Our definitions will differ slightly from the usual ones, even upon specialisation to the homogeneous scenario.

We fix $\alp, \gam \in \bR$ as given from the outset, with $\alp$ irrational and not Liouville. For $b, A_1, A_2, N_1, N_2 \in \bN$, define
\[
P(b; A_1, A_2; N_1, N_2) = \{ b + A_1 n_1 + A_2 n_2: 1 \le n_1 \le N_1, 1 \le n_2 \le N_2\} \subseteq \bN;
\]
this is a variant of a two-dimensional generalised arithmetic progression. One can read \cite[Ch. 8]{Nat1996} or \cite{TV2006} to learn more about generalised arithmetic progressions. 

Recall that the \emph{diophantine exponent} of $\alpha$ is
\[
\mu(\alp) := \sup \{w > 0: \| n \alpha \| < n^{-w} \text{ for infinitely many } n \in \bN \}.
\]
Let $\lam = 1 + \mu(\alp)$; our assumptions imply that
\[
2 \le \lam < \infty.
\]
Throughout, put
\begin{equation} \label{epsilon}
\eps = \frac1{10 \lambda}.
\end{equation}

\begin{lemma} \label{L1}
Let $N^{-2 \eps} \le \rho \le N^{- \eps}$. Then there exist $b, A_1, A_2, N_1, N_2 \in \bN$ with
\[
(A_1, A_2) = 1, \quad N_1 N_2 \asymp \rho N, \quad A_2 > N_1, \quad \min(N_1, N_2) \ge N^\eps
\]
such that $N/20 \le b \le N/2$ and
\[
P(b; A_1, A_2; N_1, N_2) \subseteq N_\gam(\alp, \rho).
\]
\end{lemma}

\begin{proof} 

Our construction relies heavily on the theory of continued fractions \cite{RS1992}. We begin by choosing $b_0 \in \bN$ such that 
\begin{equation} \label{base}
\| b_0 \alp - \gam \| \le \rho/10, \qquad b_0 \le N/10,
\end{equation}
as we now explain. We shall apply the three distance theorem \cite{MK1998} to 
\[
M := \lfloor N/10 \rfloor.
\]
This concerns the distances $d_{i+1} - d_i$, where
\[
\{ d_1, \ldots, d_M \} = \{ n \alp - \lfloor n \alp \rfloor: 1 \le n \le M \}
\]
and
\[
0 = d_0 < d_1 < \ldots < d_M < d_{M+1} = 1.
\]
For some $i \in \{0,1,\ldots, M\}$, the fractional part of $\gam$ lies in the subinterval $[d_i, d_{i+1}]$. Moreover, there exists $b_0 \in \{ 1,2,\ldots, M \}$ for which the fractional part of $b_0 \alp$ is either $d_i$ or $d_{i+1}$. Now
\[
\| b_0 \alp - \gam \| \le d_{i+1} - d_i \le d_\alp(M) := \max_{0 \le i \le M} (d_{i+1} - d_i).
\]

The three distance theorem tells us that the distances $d_{i+1} - d_i$ take on at most three distinct values, and also specifies what the three possible distances are. Let $q_1, q_2, \ldots$ be the successive denominators of the continued fraction expansion of $\alp$, and let $a_0, a_1, a_2, \ldots$ be the partial quotients. By \cite[Theorem 1]{MK1998}, there is a unique representation
\[
M = r q_k + q_{k-1} + s,
\]
where $k \in \bZ_{\ge 0}$, $1 \le r \le a_{k+1}$ and $0 \le s < q_k$. Moreover, \cite[Corollary 1]{MK1998} implies the upper bound
\[
d_\alp(M) \ll a_{k+1} \| q_k \alp \| \ll \frac{a_{k+1}}{q_{k+1}} \ll \frac1{q_k}.
\]
By \cite[Lemma 1.1]{BHV2016}, we have 
\begin{equation} \label{PowerBound}
q_{\ell+1} \ll q_\ell^\lam \qquad (\ell \in \bN),
\end{equation}
and so
\[
\| b_0 \alp - \gam \| \le d_\alp(M) \ll q_{k+1}^{-1/\lam} \ll M^{-1/\lam} \ll N^{-3\eps}.
\]
Hence $\| b_0 \alp - \gam \|  \le  \rho/10$, and we have verified \eqref{base}.

Next we tackle the homogeneous part, similarly to Tao \cite{TaoPost}. Let $t \in \bN$ be such that
\[
\rho q_{t-1}^2 \le N < \rho q_t^2,
\]
and note that
\begin{equation} \label{qtbound}
q_t > (N/\rho)^{1/2} > N^{1/2}.
\end{equation}
We use two cases, as Tao does: let
\[
(A_1, A_2) = \begin{cases}
(q_{t-2}, q_{t-1}), &\text{if } N < \rho q_{t-1} q_t \\
(q_{t-1}, q_t), &\text{if } N \ge \rho q_{t-1} q_t.
\end{cases}
\]
By a basic property of continued fractions, we have $(A_1, A_2) = 1$. Choose
\[
N_1 = \Bigl \lfloor \frac{\rho A_2}{10} \Bigr \rfloor, \quad N_2 = \Bigl \lfloor \frac N { 10 A_2} \Bigr \rfloor.
\]
Plainly $A_2 > N_1$.

If $N < \rho q_{t-1} q_t$ then, by \eqref{epsilon}, \eqref{PowerBound} and \eqref{qtbound}, we have
\[
\rho A_2 = \rho q_{t-1} \gg \rho q_t^{1/\lam} \gg \rho N^{\frac1{2\lam}} \gg N^{2 \eps}
\]
and
\[
\frac N{A_2} = \frac N{q_{t-1}} \ge \rho q_{t-1} \gg N^{2 \eps}.
\]
In particular, since $N$ is large we now have $\min(N_1, N_2) \ge N^\eps$ and $N_1 N_2 \asymp \rho N$. If, on the other hand, we have $N \ge \rho q_{t-1} q_t$, then
\[
\rho A_2 = \rho q_t \gg N^{2 \eps}
\]
and
\[
\frac N {A_2} = \frac N{q_t} \ge \rho q_{t-1} \gg N^{2 \eps}.
\]
Again we have $\min(N_1, N_2) \ge N^\eps$ and $N_1 N_2 \asymp \rho N$.

Finally, we shift our base point by defining $b = b_0 + N_2 A_2$. This lies in the correct range and, further, if $n \in P(b; A_1, A_2; N_1, N_2)$ then
\[
n \le b + N_1 A_1 + N_2 A_2 \le N/2 + \rho A_1 A_2 / 10 + N/10 \le N
\]
and
\[
\| n \alp - \gam \| \le \| b_0 \alp - \gam \| + N_1 \| A_1 \alp \| + 2N_2 \| A_2 \alp \|.
\]
If $N < \rho q_{t-1} q_t$ then 
\[
\| A_1 \alp \| \le q_{t-1}^{-1}, \qquad \| A_2 \alp \| \le q_t^{-1},
\]
so
\[
\| n \alp - \gam \| \le \frac \rho{10} + \frac \rho {10} + \frac N{5 q_{t-1} q_t} \le \rho.
\]
If instead $N \ge \rho q_{t-1} q_t$, then
\[
\| A_1 \alp \| \le q_t^{-1}, \qquad \| A_2 \alp \| \le q_{t+1}^{-1} \le q_t^{-1},
\]
so
\[
\| n \alp - \gam \| \le \frac \rho{10} + \frac \rho{10} + \frac N{5 q_t^2} \le \rho.
\]
\end{proof}

The range $N^{-2 \eps} \le \rho \le N^{- \eps}$ will suffice for our purposes. The proof could be modified so as to unify the two cases, but one aspect would become slightly more tedious, so we feel that the benefit would be marginal at best.

\section{The geometry of numbers}
\label{average}

In this section, we use the generalised arithmetic progression structure to control the average behaviour of the Euler totient function $\varphi$ on 
\[
N^*_\gam(\alp, \rho) := N_\gam(\alp, \rho) \cap [N/20, N].
\]
The AM--GM inequality \cite[Ch. 2]{Ste2004} will enable us to treat each prime separately, at which point we can employ the geometry of numbers.

\begin{lemma}
Let $N^{- 2 \eps} \le \rho \le N^{- \eps}$. Then
\[
\sum_{n \in N^*_\gam(\alp, \rho)} \frac{\varphi(n)}n \gg \rho N.
\]
\end{lemma}

\begin{proof}
Let $b,A_1,A_2,N_1$ and $N_2$ be as in Lemma \ref{L1}, and let
\[
P = P(b; A_1, A_2; N_1, N_2) \subseteq N_\gam(\alp, \rho).
\]
The lower bound $b \ge N/20$ now ensures that $P \subseteq N^*_\gam(\alp, \rho)$. We will show, \emph{a fortiori}, that
\[
\sum_{n \in P} \frac{\varphi(n)}n \gg \rho N.
\]
Note that the generalised arithmetic progression $P$ is \emph{proper}, in that if $n \in P$ then there are unique $n_1, n_2 \in \bN$ for which
\[
n_1 \le N_1, \quad n_2 \le N_2, \quad n = b + n_1 A_1 + n_2 A_2.
\]
This property follows from the fact that $(A_1,A_2) = 1$, together with the fact that $A_2 > N_1$. Now
\[
|P| = N_1 N_2 \asymp \rho N
\]
so, by the AM--GM inequality, it suffices to prove that
\[
X := \Bigl( \prod_{n \in P} \frac{\varphi(n)}n \Bigr) ^{|P|^{-1}} \gg 1.
\]

Next we use the identity 
\[
\frac{\varphi(n)}n = \prod_{p \mid n} (1-1/p),
\]
and swap the order of multiplication. This gives
\[
X = \prod_p (1- 1/p)^{\alp_p},
\]
where
\[
\alp_p = |P|^{-1} \# \{ n \in P: n \equiv 0 \mmod p \}
\]
should be regarded as a proportion. 

This leads us to bound the quantities $\alp_p$. To this end, we may suppose that $\alp_p > 0$, whereupon $p \le N$. Fix positive integers $n_1^* \le N_1$ and $n_2^* \le N_2$ such that
\[
b+ A_1 n_1^* + A_2 n_2^* \equiv 0 \mmod p.
\]
The congruence
\[
b + A_1 n_1 + A_2 n_2 \equiv 0 \mmod p
\]
then implies that
\begin{equation} \label{cong}
A_1 n'_1 +  A_2 n'_2 \equiv 0 \mmod p,
\end{equation}
where $n'_1 = n_1 - n_1^*$ and $n'_2 = n_2 - n_2^*$ are integers in the box
\[
\cB := [-N_1, N_1] \times [-N_2, N_2] \subseteq \bR^2.
\]
In particular, the quantity $|P| \alp_p$ is bounded above by the number of integer solutions to $\eqref{cong}$ in the box $\cB$.

Recall that $(A_1, A_2) = 1$. If $p \mid A_1$ or $p \mid A_2$ then 
\[
\alp_p \ll \frac{ N_1(1+N_2/p) + N_2(1+N_1/p)}{N_1N_2} \ll p^{-1} + N_1^{-1} + N_2^{-1}.
\]
As $\min(N_1,N_2) \ge N^\eps \ge p^\eps$, we get $\alp_p \ll p^{-\eps}$ in this case.

Now assume that $A_1$ and $A_2$ are not divisible by $p$. Then \eqref{cong} defines a full lattice of determinant $p$. We apply the following special case of a counting lemma \cite[Lemma 2]{Sch1995}.

\begin{lemma} \label{SchmidtLemma}
Let $\cS \subseteq \bR$ be a convex set containing the origin, and suppose that $\cS$ lies in the compact disc of radius $r$ centred at the origin. Let $V(\cS)$ denote the volume of $\cS$, and let $\Lam \subseteq \bZ^2$ be a full lattice in $\bR^2$. Then
\[
|\cS \cap \Lam| = 1 + V(S) / \det(\Lam) + O(r).
\]
\end{lemma}

By Lemma \ref{SchmidtLemma}, we have
\[
|P| \alp_p \ll \frac{N_1N_2}p + N_1 + N_2
\]
and, as $|P| = N_1N_2$, we again get
\[
\alp_p \ll p^{-1} + N_1^{-1} + N_2^{-1} \ll p^{-\eps}.
\]

Setting
\[
Y = \log(1/X),
\]
we find that
\[
Y \le \sum_p \alp_p \log(1+2/p) \ll \sum_p p^{-\eps} \log(1+2/p).
\]
As $\log(1+2/p) \le 2/p$, we now have
\[
Y \ll \sum_p p^{-1-\eps} \ll 1,
\]
so $X \gg 1$. As discussed, this completes the proof.
\end{proof}

We also record the corresponding upper bound, for later use. This will be a fairly simple consequence of \cite[Lemmas 6.1 and 6.3]{BHV2016}, which we state below for convenience. 

\begin{lemma} \label{L61}
Let $\alp \in \bR \setminus \bQ$, and let $q_\ell$ $(\ell = 0,1,2,\ldots)$ be the successive denominators of the continued fraction expansion of $\alp$. Let $N \in \bN$, and let $\xi > 0$ be such that $0 < 2 \xi < \| q_2 \alp \|$. Suppose
\[ 
\frac1{2\xi} \le q_\ell \le N
\]
for some $\ell$. Then
\[
\lfloor \xi N \rfloor \le \# N_0(\alp, \xi) \le 32 \xi N.
\]
\end{lemma}

The following is the first part of \cite[Lemma 6.3]{BHV2016}.

\begin{lemma} \label{L63} For $\alp \in \bR$, $\rho > 0$ and $N \in \bN$, we have
\[
\# N_\gam(\alp, \rho) \le 1 + \#N_0(\alp, 2 \rho).
\]
\end{lemma}

Let $\eta$ be a positive constant, small in terms of $\alp$. When
\begin{equation} \label{GoodRange}
N^{-4 \eps} \le \rho \le \eta,
\end{equation}
we see from \eqref{epsilon} and \eqref{PowerBound} that some continued fraction denominator $q_\ell$ must lie in the range $[(4 \rho)^{-1}, N]$. We may therefore apply Lemmas \ref{L61} and \ref{L63} to get
\begin{equation} \label{UpperCount}
\# N_\gam(\alp, \rho) \ll \rho N.
\end{equation}
Now a trivial estimate yields
\[
\sum_{n \in N^*_\gam(\alp, \rho)} \frac{\varphi(n)}n \le \sum_{n \in N_\gam(\alp, \rho)} \frac{\varphi(n)}n \ll \rho N.
\]
We summarise the primary outcome of this section as follows.

\begin{cor} \label{RightSize}
Let $N^{-2 \eps} \le \rho \le N^{- \eps}$. Then
\[
\sum_{n \in N^*_\gam(\alp, \rho)} \frac{\varphi(n)}n \asymp \rho N.
\]
The implicit constants depend at most on $\alp$.
\end{cor}

\section{Sums of reciprocals of fractional parts}
\label{fractional}

As in \cite{BHV2016}, an essential part of the analysis is to estimate inhomogeneous variants of sums of reciprocals of fractional parts. Recall that we fixed real numbers $\alp$ and $\gam$, with $\alp$ irrational and not Liouville, from the beginning. A restriction on the range of summation will enable us to go beyond what was shown in \cite{BHV2016}. Let
\[
B = \{ n \in \bN: \| n \alp - \gam \| < n^{-4 \eps} \}.
\]
We consider the sums
\[
T_N(\alp, \gam) := \sum_{\substack{n \le N \\ n \notin B}} \frac1{\| n \alp - \gam\|}, \qquad  T_N^*(\alp, \gam) := \sum_{\substack{n \le N \\ n \notin B}} \frac{\varphi(n)}{n \| n \alp - \gam\|},
\]
and show that
\begin{equation} \label{order}
T_N(\alp, \gam) \asymp  T_N^*(\alp, \gam)  \asymp N \log N.
\end{equation}

We begin with an upper bound.

\begin{lemma} \label{upper} We have
\[
T_N(\alp, \gam) \ll N \log N.
\]
\end{lemma}

\begin{proof}
Let $\eta$ be as it is in \eqref{GoodRange}. Now
\[
T_N(\alp, \gam) \le T_1 + T_2,
\]
where
\[
T_1 = \sum_{\substack{n \le N \\ N^{-4\eps}      \le        \| n \alp - \gam \|   \le \eta}} \frac1{\| n \alp - \gam\|}, \\
\]
and
\[
T_2 = \sum_{\substack{n \le N \\ \eta    <        \| n \alp - \gam \|   \le 1}} \frac1{\| n \alp - \gam\|}.
\]
A trivial estimate gives $T_2 \ll N$, so it remains to show that
\begin{equation} \label{remains}
T_1 \ll N \log N.
\end{equation}

We can decompose $T_1$ into at most $O(\log N)$ sums of the form
\[
\sum_{n \in N_\gam(\alp, \rho) \setminus N_\gam(\alp, \rho/2)} \frac1{ \| n \alp - \gam \|}.
\]
For each, we have \eqref{GoodRange} and therefore \eqref{UpperCount}. Hence
\[
\sum_{n \in N_\gam(\alp, \rho) \setminus N_\gam(\alp, \rho/2)} \frac1{ \| n \alp - \gam \|} \ll
\frac{ |N_\gam(\alp, \rho)| } {\rho} \ll N.
\]
Since there are $O(\log N)$ such ranges at most, we obtain \eqref{remains}, completing the proof.
\end{proof}

We also require a lower bound for $T_N^*(\alp, \gam)$. 

\begin{lemma} \label{lower} We have
\[
T_N^*(\alp, \gam) \gg N \log N.
\]
\end{lemma}

\begin{proof}
First observe that if $N/20 \le n \le N$ and $\| n \alp - \gam \| \ge N^{-2 \eps}$ then $n \notin B$. It therefore suffices to prove that
\begin{equation} \label{suffices}
\sum_{\substack{N/20 \le n \le N \\ 
N^{-2 \eps} \le  \| n \alp - \gam \| \le N^{-\eps}} } \frac{\varphi(n)}{n \| n \alp - \gam\|} \gg N \log N.
\end{equation}

Let $\del$ be a positive constant, small in terms of $\alp$. There are at least a constant times $\log N$ disjoint subintervals $(\del \rho, \rho]$ of the interval $[N^{-2 \eps}, N^{-\eps}]$, and for each we have the bound
\begin{align*}
\sum_{n \in N^*_\gam(\alp, \rho) \setminus N^*_\gam(\alp, \del \rho)} \frac{\varphi(n)} {n \| n \alp - \gam \|} 
&\ge \rho^{-1}  \sum_{n \in N^*_\gam(\alp, \rho) \setminus N^*_\gam(\alp, \del \rho)} \frac{\varphi(n)} n \\
&= \rho^{-1} \Biggl( \sum_{n \in N^*_\gam(\alp, \rho)} \frac{\varphi(n)} n -  \sum_{n \in N^*_\gam(\alp, \del \rho)} \frac{\varphi(n)} n \Biggr).
\end{align*}
As $\del$ is small in terms of the constants implicit in Corollary \ref{RightSize}, it now follows readily from Corollary \ref{RightSize} that
\[
\sum_{n \in N^*_\gam(\alp, \rho) \setminus N^*_\gam(\alp, \del \rho)} \frac{\varphi(n)} {n \| n \alp - \gam \|} \gg N.
\]
We have this for each of the disjoint subintervals, and so we have \eqref{suffices}.
\end{proof}

Note from the definitions that $T_N(\alp, \gam) \ge T_N^*(\alp, \gam)$. By Lemmas \ref{upper} and \ref{lower}, we now know \eqref{order}.

\section[The Duffin--Schaeffer theorem]{An application of the Duffin--Schaeffer theorem}
\label{DS}

In this section, we finish the proof of Theorem \ref{main}. The overall strategy is to apply the Duffin--Schaeffer theorem (Theorem \ref{DSthm}) to the approximating function
\[
\Psi(n) = \Psi_\alp^\gam(n) = \begin{cases} \frac{\psi(n)}{\| n \alp - \gam \|}, &\text{if }  n \notin B \\
0, &\text{if } n \in B.
\end{cases}
\]
A valid application of the Duffin--Schaeffer theorem will complete the proof, so we need only verify its hypotheses, namely
\begin{equation} \label{hyp1}
\sum_{n=1}^\infty \frac {\varphi(n)}n \Psi(n) = \infty
\end{equation}
and
\begin{equation} \label{hyp2}
\sum_{n \le N} \frac {\varphi(n)}n \Psi(n) \gg \sum_{n \le N} \Psi(n).
\end{equation}
The inequality \eqref{hyp2} is only needed for an infinite strictly increasing sequence of positive integers $N$, but we shall prove \emph{a fortiori} that for all large $N$ we have
\begin{equation} \label{hyp21}
\sum_{n \le N} \frac {\varphi(n)}n \Psi(n)  \gg  \sum_{n \le N} \psi(n) \log n
\end{equation}
and
\begin{equation} \label{hyp22}
\sum_{n \le N} \Psi(n) \ll \sum_{n \le N} \psi(n) \log n.
\end{equation}
Observe, moreover, that \eqref{divergence} and \eqref{hyp21} would imply \eqref{hyp1}. The upshot is that it remains to prove \eqref{hyp21} and \eqref{hyp22}.

Recall the sums $T_N(\alp, \gam)$ and $T_N^*(\alp, \gam)$ from \S \ref{fractional}, and let $N_0 \in \bN$ be a large constant. By partial summation \cite[Eq. (1.2)]{BHV2016}, and using the fact that 
\[
\psi(n) \ge \psi(n+1),
\]
we have the lower bound
\[
\sum_{n \le N}  \frac {\varphi(n)}n \Psi(n) \ge
\psi(N+1) T_N^*(\alp, \gam) + \sum_{n = N_0}^N (\psi(n) - \psi(n+1)) T_n^*(\alp,\gam).
\]
Applying Lemma \ref{lower} to continue our calculation yields
\[
\sum_{n \le N}  \frac {\varphi(n)}n \Psi(n) \gg 
\psi(N+1) N \log N + \sum_{n = N_0}^N (\psi(n) - \psi(n+1)) n \log n.
\]
As $\psi(n) \ge \psi(n+1)$ and $\sum_{m \le n} \log m \le n \log n$, we now have
\[
\sum_{n \le N}  \frac {\varphi(n)}n \Psi(n) \gg \psi(N+1) \Bigl(\sum_{m \le N} \log m \Bigr) + \sum_{n = N_0}^N (\psi(n) - \psi(n+1)) \sum_{m \le n} \log m.
\]
Another application of partial summation now gives
\[
\sum_{n \le N}  \frac {\varphi(n)}n \Psi(n) \gg \sum_{n = N_0}^N \psi(n) \log n,
\]
establishing \eqref{hyp21}.

We come to the final piece of the puzzle, which is \eqref{hyp22}. By partial summation, we have
\[
\sum_{n \le N}  \Psi(n)  = \psi(N+1)T_N(\alp, \gam)
+ \sum_{n \le N} (\psi(n) - \psi(n+1)) T_n(\alp,\gam).
\]
Observe that if $n \le N_0$ then $T_n(\alp, \gam) \le T_{N_0}(\alp, \gam) \ll 1$. Thus, applying Lemma \ref{upper} to continue our calculation yields
\[
\sum_{n \le N}  \Psi(n) \ll 1 + \psi(N+1)N \log N + \sum_{n = N_0}^N (\psi(n) - \psi(n+1)) n \log n.
\]
Partial summation tells us that $\sum_{m \le n} \log m \gg n \log n$, and so
\[
\sum_{n \le N}  \Psi(n) \ll 1 + \psi(N+1) \Bigl(\sum_{m \le N} \log m \Bigr) + \sum_{n = N_0}^N (\psi(n) - \psi(n+1)) \sum_{m \le n} \log m.
\]
A further application of partial summation now gives
\[
\sum_{n \le N}  \Psi(n) \ll 1 + \sum_{n = N_0}^N \psi(n) \log n.
\]
This confirms \eqref{hyp22}, thereby completing the proof of Theorem \ref{main}.


\providecommand{\bysame}{\leavevmode\hbox to3em{\hrulefill}\thinspace}

\end{document}